\documentclass[letter, 10 pt, conference]{ieeeconf}  

\IEEEoverridecommandlockouts                              

\usepackage{amsmath} 

\usepackage{gensymb}
\usepackage{amsthm}
\usepackage{amssymb}  
\usepackage{stfloats}
\usepackage[noadjust]{cite}
\usepackage{graphicx}
\graphicspath{{./pics/}}

\theoremstyle{definition}
\newtheorem{definition}{Definition}[section]
\newtheorem{theorem}{Theorem}[section]
\newtheorem{corollary}{Corollary}[theorem]
\newtheorem{lemma}[theorem]{Lemma}
\theoremstyle{remark}
\newtheorem*{remark}{Remark}

\title{\LARGE \bf
	Negative Imaginary State Feedback Control with a Prescribed Degree of Stability
}

\author{James Dannatt$^{1}$ and Ian Petersen$^{2}$
\thanks{*This work was supported by the Australian Research Council under grant DPI60101121.}
\thanks{$^{1}$James Dannatt is with the Research School of Engineering at the Australian National University, ACT.
        {\tt\small james.dannatt@anu.edu.au}}%
\thanks{$^{2}$Ian Petersen is with the Research School of Engineering at the Australian National University, ACT.
        {\tt\small i.r.petersen@gmail.com}}%
}

\begin{document}

\maketitle
\thispagestyle{empty}
\pagestyle{empty}

\begin{abstract}
This paper presents a method for the synthesis of negative imaginary closed-loop systems with a prescribed degree of stability under the assumption of full state feedback. The approach extends existing work by using a perturbation method to ensure a closed-loop system that has both the negative imaginary property and a prescribed degree of stability. This approach involves the real Schur decomposition of a matrix followed by the solution to two Lyapunov equations which provides computational advantages over alternate state feedback synthesis techniques. Also, some counterexamples are presented which clarify the perturbation properties of strictly negative imaginary systems. Finally, an illustrative example demonstrates the approach.
\end{abstract}

\section{Introduction}

Negative imaginary (NI) systems theory is concerned with stable systems that have a phase response in the interval $[-\pi, 0]$ for positive frequencies~\cite{Lanzon2008}. This corresponds to a positive real (PR) frequency response rotated clockwise by $90\degree$ in the Nyquist plane. However, NI systems theory is more than a simple rotation of PR theory. Whereas PR systems are constrained to have a relative degree of either zero, one or minus one, NI systems theory supports systems with a relative degree of zero, one and two \cite{Petersen2010}. This has resulted in NI theory emerging as a flourishing complement to positive real (PR) and passivity theory~\cite{Lanzon2008,Petersen2010}.
\\\\
NI systems theory was originally motivated by the study of linear mechanical systems with collocated force inputs and position outputs~\cite{Lanzon2008}. However, NI systems theory can be applied in many other domains, such as RLC circuits~\cite{Petersen2015}, nano-positioning systems~\cite{Mabrok2014} and formation control of multiple UAVs~\cite{8287487}. Robust stability conditions for NI systems have been developed and are well understood \cite{Lanzon2008,5483104}. In particular, it has been shown that the positive feedback interconnection of an NI system with a strictly negative imaginary (SNI) system is internally stable as long as conditions on the closed-loop DC gain are satisfied \cite{Lanzon2008}. It is this understanding of NI robust stability conditions that has motivated controller synthesis results with the aim of creating a closed-loop system with the NI or SNI property. This closed-loop NI property would for example guarantee robust stability of the closed-loop to the un-modeled dynamics of a flexible structure~\cite{7074636}.
\\\\
Early work on controller synthesis within the NI framework was presented in \cite{Petersen2010} and \cite{5430931} with a focus on state feedback and linear matrix inequality (LMI) based synthesis techniques. Drawing on the $H_{\infty}$ literature, \cite{Mabrok2012} proposed a synthesis approach using the solution to an algebraic Riccati equation (ARE) that could be obtained by solving two Lyapunov equations. This approach was computationally efficient and scaled well with high order systems but left a closed-loop pole at the origin ensuring a marginally stable closed-loop system.
\\\\
Realizing this shortcoming, the papers \cite{Mabrok2012a,Mabrok2015} modified the approach of \cite{Mabrok2012} using a perturbation applied to the plant matrix of the open-loop system in order to ensure asymptotic stability of the closed-loop system. The perturbation achieved closed-loop asymptotic stability but the closed-loop system could no longer be guaranteed to be NI as no proof was offered to support the preservation of the NI property.
\\\\
Since the publishing of \cite{Mabrok2012a} and \cite{Mabrok2015,FERRANTE20132138} translated the rational positive real property of \cite{anderson2013network} to show that a symmetric negative imaginary transfer function is negative imaginary over an entire orthant of interest.
\\\\
This paper extends the work of \cite{Mabrok2012a} and \cite{Mabrok2015} by showing that a system preserves the NI property after a negative perturbation of the plant matrix if and only if it is NI over a specific orthant in the complex plane. We show that the orthant of interest formed from plant matrix perturbation is exactly the orthant described in \cite{FERRANTE20132138} in the single input single output (SISO) case. These results are then brought together to give a new NI synthesis method that guarantees both stability and the NI property of the closed-loop system. Furthermore, the proposed approach guarantees a prescribed degree of stability in addition to asymptotic stability e.g. see \cite{5249185}. This is useful in designing a control system to achieve not only robustness but also an adequate level of performance. Our approach offers the same computational advantages as the technique given in \cite{Mabrok2012} as it requires only the Schur decomposition of a matrix and the solution of two Lyapunov equations.

\section{Definitions}

The notation $Im[G(j\omega)]$ refers to the imaginary component of the frequency response $G(j\omega)$. Analogously $Re[G(j\omega)]$ refers to the real component of $G(j\omega)$. $C^*$ refers to the complex conjugate transpose of a matrix or vector $C$.

\section{Preliminaries}

Consider the linear time-invariant (LTI) system,
\begin{IEEEeqnarray}{c} \label{system}
	\dot{x}(t) = Ax(t) + Bu(t), \nonumber \\
	y(t) = Cx(t) + Du(t)
\end{IEEEeqnarray}
where $A \in \mathbb{R}^{n \times n},$ $B \in \mathbb{R}^{n \times m},$ $C \in \mathbb{R}^{m \times n}$ and $D \in \mathbb{R}^{m \times m}$.

The following two definitions relate to the NI and SNI properties of the transfer function matrix $G(s) = C(sI - A)^{-1}B+D$ corresponding to the system (\ref{system}).

\begin{remark}
The following definitions hold for general MIMO system. However, the main result of this paper applies strictly to SISO systems where $m=1$.
\end{remark}

\theoremstyle{definition}
\begin{definition} \label{def:NI}{ A square transfer function matrix $G(s)$ is NI if the following conditions are satisfied \cite{8287451}:}
	\begin{enumerate}
		\item G(s) has no pole in $Re[s]>0$.

		\item For all $\omega \geq 0$ such that $jw$ is not a pole of $G(s)$, $j(G(j\omega) - G(j\omega)^*) \geq 0$.

		\item If $s=j\omega_0$, $ \omega_0 > 0$ is a pole of $G(s)$ then it is a simple pole. Furthermore, if $s=j\omega_0$, $ \omega_0 > 0$ is a pole of $G(s)$, then the residual matrix $K = \lim_{s \to j\omega_0} (s-j\omega_0)jG(s)$ is positive semidefinite Hermitian.

		\item If $s=0$ is a pole of $G(s)$, then it is either a simple pole or a double pole. If it is a double pole, then, $\lim_{s \to 0} s^2G(s) \geq 0$.
	\end{enumerate}
\end{definition}

Also, an LTI system (\ref{system}) is said to be NI if the corresponding transfer function matrix $G(s) = C(sI-A)^{-1}B + D$ is NI.

\begin{definition} \label{def: SNI freq domain definition}
{A square transfer function matrix $G(s)$ is SNI if the following conditions are satisfied \cite{8287451}:}
	\begin{enumerate}
		\item G(s) has no poles in $Re[s] \geq 0$.

		\item For all $\omega > 0$ such that $jw$ is not a pole of $G(s)$, $j(G(j\omega) - G(j\omega)^*) > 0$.
	\end{enumerate}
\end{definition}

Also, an LTI system (\ref{system}) is said to be SNI if the corresponding transfer function matrix $G(s) = C(sI-A)^{-1}B + D$ is SNI.
\\\\
The above NI and SNI definitions provide a means of testing if a system is NI or SNI by analyzing the system properties in the frequency domain. Both definitions will be used in the proofs to follow.
\\\\
An alternate method of characterizing NI or SNI systems that is directly used in controller synthesis is provided by the following Riccati equation based lemmas.

\begin{lemma}\label{theorem:NI using riccati} (\cite{Mabrok2015})
Let $
\begin{bmatrix}
    \begin{tabular}{ l | r }
  A & B \\ \hline
  C & D
\end{tabular}
\end{bmatrix}
$ be a minimal realization of $G(s)$ and suppose $CB + B^TC^T > 0$. Then $G(s)$ is NI if and only if $D = D^T$ and there is exists a matrix $P \geq 0$ which solves the following algebraic Riccati equation
  \begin{IEEEeqnarray}{c} \label{math:riccati}
  PA_0 + {A_0}^TP + PBR^{-1}B^TP + Q = 0,
  \end{IEEEeqnarray}
where
\begin{IEEEeqnarray*}{c}
  A_0 = A-BR^{-1}CA, \\
    R = CB + B^TC^T, \\
    Q = A^TC^TR^{-1}CA
  \end{IEEEeqnarray*}
  and $A_0 \in \mathbb{R}^{n \times n}$, $R \in \mathbb{R}^{m \times m}$ and $Q \in \mathbb{R}^{n \times n}$.
\end{lemma}
\begin{proof}
A proof of this lemma is given in \cite{Mabrok2015}.
\end{proof}

The corresponding lemma for the characterization of SNI transfer functions is as follows.

\begin{lemma} \label{lemma: SNI Riccati conditions}
Let $
\begin{bmatrix}
    \begin{tabular}{ l | r }
  A & B \\ \hline
  C & D
\end{tabular}
\end{bmatrix}
$ be a minimal realization of $G(s)$ and suppose $CB + B^TC^T > 0$. Then $G(s)$ is SNI if and only if
\begin{enumerate}
\item $A$ has no imaginary-axis eigenvalues and $D = D^T$

\item There is exists a matrix $P=P^T>0$ which solves the following algebraic Riccati equation and $P$ is a stabilizing solution.
  \begin{IEEEeqnarray}{c} \label{math:riccati SNI}
  PA_0 + {A_0}^TP + PBR^{-1}B^TP + Q = 0,
  \end{IEEEeqnarray}
where
\begin{IEEEeqnarray*}{c}
    A_0 = A-BR^{-1}CA, \\
    R = CB + B^TC^T, \\
    Q = A^TC^TR^{-1}CA
  \end{IEEEeqnarray*}
  and $A_0 \in \mathbb{R}^{n \times n}$, $R \in \mathbb{R}^{m \times m}$ and $Q \in \mathbb{R}^{n \times n}$.

 \item All the eigenvalues of the matrix $A_0 + BR^{-1}B^TP$ lie in the open left half of the complex plane or at the origin.
\end{enumerate}
\end{lemma}
\begin{proof}
This result is an extension of Theorem 3 of \cite{Mabrok2012a} and the proof follows closely the proof given in \cite{Mabrok2012a}.
\\\\
Suppose that $G(s)$ is SNI. Then conditions 1), 2) and 3) follow as in the proof given in \cite{Mabrok2012a}. Conversely, suppose that conditions 1), 2) and 3) are satisfied. It follows from Condition 2) that a $P>0$ exists such that $G(s)$ is NI and
\begin{align*}
  PA + A^TP \leq 0,
\end{align*}
is satisfied. Therefore $A$ is Lyapunov stable and it follows from Condition 1) that $G(s)$ is actually SNI.
\end{proof}

\begin{lemma} \label{lemma: NI Ricc stabilizing always singular
} Suppose A is non-singular and $R = CB + B^TC^T > 0.$ If there is exists a matrix $P=P^T>0$ which solves (\ref{math:riccati SNI}) and the eigenvalues of the matrix $A_0 + BR^{-1}B^TP$ are in the closed left half of the complex plane, then $A_0 + BR^{-1}B^TP$ will always be singular.
\end{lemma}
\begin{proof}
The Hamiltonian matrix associated with (\ref{math:riccati SNI}) is given by,
\begin{align*}
  H = \begin{bmatrix}
        \begin{tabular}{ l r }
        $A-BR^{-1}CA$     & $BR^{-1}B^T$  \\
        $-A^TC^TR^{-1}CA$   & $-A^T+A^TC^TR^{-1}B^T$
      \end{tabular}
    \end{bmatrix}.
\end{align*}

Hence, the matrix $A_0 + BR^{-1}B^TP$ will be singular if there exists a vector $\begin{bmatrix}
      \begin{tabular}{ l }
        x \\
        y
    \end{tabular}
  \end{bmatrix}$ such that
\begin{align*}
  \begin{bmatrix}
      \begin{tabular}{ l r }
        $A-BR^{-1}CA$     & $BR^{-1}B^T$  \\
        $-A^TC^TR^{-1}CA$   & $-A^T+A^TC^TR^{-1}B^T$
    \end{tabular}
  \end{bmatrix}
  \begin{bmatrix}
      \begin{tabular}{ l }
        x \\
        y
    \end{tabular}
  \end{bmatrix} = 0.
\end{align*}
This is equivalent to
\begin{align} \label{eq_1}
  (I-BR^{-1}C)Ax  + BR^{-1}B^Ty = 0,
\end{align}
\begin{align}  \label{eq_2}
-A^TC^TR^{-1}CAx - A^T(I+C^TR^{-1}B^T)y = 0.
\end{align}
Now let $z = Ax$. Equations (\ref{eq_1}) and (\ref{eq_2}) are equivalent to,
\begin{align}
  z-BR^{-1}Cz  + BR^{-1}B^Ty &= 0,  \\
C^TR^{-1}Cz + y -C^TR^{-1}B^Ty &= 0.
\end{align}
In matrix form this is equivalent to
\begin{align*}
  \begin{bmatrix}
      \begin{tabular}{ l r }
        $I-BR^{-1}C$    & $BR^{-1}B^T$  \\
        $C^TR^{-1}C$  & $I-C^TR^{-1}B^T$
    \end{tabular}
  \end{bmatrix}
  \begin{bmatrix}
      \begin{tabular}{ l }
        z \\
        y
    \end{tabular}
  \end{bmatrix} = 0,
\end{align*}
\begin{align*}
   \iff I - \begin{bmatrix}
      \begin{tabular}{ l r }
        $BR^{-1}C$    & $-BR^{-1}B^T$   \\
        $-C^TR^{-1}C$   & $C^TR^{-1}B^T$
    \end{tabular}
  \end{bmatrix}
  \begin{bmatrix}
      \begin{tabular}{ l }
        z \\
        y
    \end{tabular}
  \end{bmatrix} = 0,
\end{align*}
\begin{align} \label{math: expanded equation 1}
  \iff  I - \begin{bmatrix}
      \begin{tabular}{ l  }
        $B$     \\
        $-C^T$
    \end{tabular}
  \end{bmatrix} R^{-1}
  \begin{bmatrix}
      \begin{tabular}{l r}
        C  & $-B^T$
    \end{tabular}
  \end{bmatrix}
  \begin{bmatrix}
      \begin{tabular}{ l }
        z \\
        y
    \end{tabular}
  \end{bmatrix} = 0.
\end{align}
Here,
\begin{align*}
  R = CB + B^TC^T &= \begin{bmatrix}
      \begin{tabular}{ l  r}
        $C$ &  $-B^T$
    \end{tabular}
  \end{bmatrix}
  \begin{bmatrix}
      \begin{tabular}{l r}
        B \\
         $-C^T$
    \end{tabular}
  \end{bmatrix}, \\
  &= V_2^TV_1 > 0,
\end{align*}
where
\begin{align*}
  V_1 &=  \begin{bmatrix}
      \begin{tabular}{ l  r}
        $B$ \\
        $-C^T$
    \end{tabular}
  \end{bmatrix}, &
  V_2 &=
  \begin{bmatrix}
      \begin{tabular}{l r}
        $C^T$ \\
         $-B$
    \end{tabular}
  \end{bmatrix} .
\end{align*}
Hence, (\ref{math: expanded equation 1}) is equivalent to
\begin{align} \label{math: v equation 1}
  (I - V_1(V_2^TV_1)^{-1}V_2^T) \begin{bmatrix}
      \begin{tabular}{ l }
        z \\
        y
    \end{tabular}
  \end{bmatrix}.
\end{align}
Thus, we wish to construct a non zero vector $w := \begin{bmatrix}
      \begin{tabular}{ l }
        z \\
        y
    \end{tabular}
  \end{bmatrix}$ such that
\begin{align}
  (I - V_1(V_2^TV_1)^{-1}V_2^T)w = 0.
\end{align}
We now construct $w$ of the form $w = V_1 \beta$ and choose a $\beta$ such that $w=V_1 \beta \neq 0$. Then
\begin{align*}
& V_1 (\beta - (V_2^TV_1)^{-1}V_2^TV_1 \beta) = 0, \\
\implies& V_1 \beta - V_1(V_2^TV_1)^{-1}V_2^TV_1 \beta = 0, \\
\implies& w - V_1(V_2^TV_1)^{-1}V_2^Tw = 0, \\
\implies& (I - V_1(V_2^TV_1)^{-1}V_2^T)w = 0.
\end{align*}
This implies that $A_0 + BR^{-1}B^TP$ is singular.
\end{proof}

The Riccati equation SNI lemma implies the existence of poles at the origin in the stabilizing solution of (\ref{math:riccati SNI}). It follows that the corresponding Hamiltonian matrix is singular and as both Lemma~\ref{theorem:NI using riccati} and Lemma~\ref{lemma: SNI Riccati conditions} share the same Riccati equation, this is actually true of both Riccati equation NI lemmas.
\\\\
The significance of this result comes from the fact that AREs with singular Hamiltonians can be computationally difficult to solve~\cite{bini2012numerical}. For small order systems, this doesn't cause an immediate problem. However for large order systems, this can impact results that depend on solving the AREs (\ref{math:riccati}) and (\ref{math:riccati SNI}). The NI controller synthesis method presented in \cite{Mabrok2012} is one such result that suffers from this problem as it uses the ARE (\ref{math:riccati}) to construct a controller.
\\\\
We will now present the NI controller synthesis lemma offered in \cite{Mabrok2012} and highlight the techniques used to address the computational difficulties associated with the singular Hamiltonian of (\ref{math:riccati}).
\\\\
First, consider the following state space representation of a linear uncertain system with SNI uncertainty:
\begin{align} \label{uncertain system}
  \dot{x} &= Ax + B_1w + B_2u; \\
  z       &= C_1x;             \\
  w       &= \Delta(s)z. \label{uncertain systemc}
\end{align}
where $A \in \mathbb{R}^{n \times n}$, $B_1 \in \mathbb{R}^{n \times 1}$, $B_2 \in \mathbb{R}^{n \times r}$, $C_1 \in \mathbb{R}^{1 \times n}$, $\Delta(s)$ represents an uncertainty transfer function matrix which is assumed to be SNI, $C_1B_2$ is non-singular and $R = C_1B_1 + B_1^TC_1^T > 0$.
\\\\
If we apply a state feedback controller $u=Kx$ to this system, the corresponding closed-loop uncertain system is given by
\begin{align} \label{math: closed loop system}
  \dot{x} &= (A + B_2K)x + B_1w; \\
  z       &= C_1x;        \label{math: closed loop systemb}     \\
  w       &= \Delta(s)z.
\end{align}

A technique for constructing the required controller matrix $K$ is given in \cite{Mabrok2012a} and \cite{Mabrok2015}. It constructs the controller such that the closed-loop system (\ref{math: closed loop system}), (\ref{math: closed loop systemb}) has the negative imaginary property.
\\\\
Consider the following real Schur transformation of the matrix $A-B_2(C_1B_2)^{-1}C_1A$ which is applied to the system (\ref{math: closed loop system}), (\ref{math: closed loop systemb}):

\begin{align*}
  A_f &= U^T(A-B_2(C_1B_2)^{-1}C_1A)U   =
  \begin{bmatrix}
    \begin{tabular}{ l  r }
  $A_{11}$ & $A_{12}$ \\
  0 & $A_{22}$
\end{tabular}
\end{bmatrix},
\\
  B_f &= U^T(B_2(C_1B_2)^{-1} - B_1R^{-1}) =
  \begin{bmatrix}
    \begin{tabular}{c}
      $B_{f1}$ \\
      $B_{f1}$
\end{tabular}
\end{bmatrix},\\
  \tilde{B_1} &= U^TB_1 =
  \begin{bmatrix}
    \begin{tabular}{c}
      $B_{11}$ \\
      $B_{22}$
\end{tabular}
\end{bmatrix},
\end{align*}
such that all of the eigenvalues of the matrix $A_{11}$ are in the closed left half plane and $A_{22}$ is an anti-stable matrix.

\begin{theorem}\label{lemma: mabrok synthesis lemma} (\cite{Mabrok2012,Mabrok2015}) Consider the uncertain system (\ref{uncertain system})-(\ref{uncertain systemc}) with $C_1B_2$ invertible and $R = C_1B_1 + B_1^TC_1^T > 0$. Then there exists a controller K such that the closed-loop system (\ref{math: closed loop system}), (\ref{math: closed loop systemb}) is NI if there exist matrices $T \geq 0$ and $S \geq 0$ such that

\begin{align}
  -A_{22}T - TA_{22}^T + B_{f2}RB_{f2}^T &= 0, \label{math: mabrok lyap 1}\\
  -A_{22}S - SA_{22}^T + B_{22}R^{-1}B_{22}^T &= 0 \label{math: mabrok lyap 2}
\end{align}
and $S - T < 0$. Furthermore, the required controller gain matrix is given by

\begin{align*}
K &= (C_1B_2)^{-1}(B_1^TP - C_1A - R(B_2^TC_1^T)^{-1}B_2^TP),
\end{align*}
where $P = UP_fU^T$ and $P_f = \begin{bmatrix}
    \begin{tabular}{ l  r }
  0 & 0 \\
  0 & $(T-S)^{-1}$
\end{tabular}
\end{bmatrix} \geq 0$. Also, the matrix $P_f$ satisfies the algebraic Riccati equation
\begin{align} \label{math: AF PF riccati euqation}
P_fA_f + A_f^TP_f - P_fB_fRB_f^TP_f + P_f\tilde{B_1}R^{-1}\tilde{B_1^T}P_f = 0.
\end{align}
\end{theorem}
\begin{proof}
A proof of this lemma is given in \cite{Mabrok2012a} and \cite{Mabrok2015}.
\end{proof}

It is clear from Theorem~\ref{lemma: mabrok synthesis lemma} that the ARE (\ref{math: AF PF riccati euqation}) does not need to be solved directly. Rather than potentially dealing with the computational challenges presented by the singular Hamiltonian, the solution $P_f$ is obtained by solving the two Lypanov equations (\ref{math: mabrok lyap 1}), (\ref{math: mabrok lyap 2}). This gives Theorem~\ref{lemma: mabrok synthesis lemma} significant computational advantages over other synthesis techniques as the system order increases.
Despite the computational advantages associated with Theorem~\ref{lemma: mabrok synthesis lemma}, this approach has an inherent problem highlighted in the following corollary.

\begin{corollary} \label{corollary: closed-loop synthesises system has origin pole}
  The controller synthesized using Theorem~\ref{lemma: mabrok synthesis lemma} will always result in a closed-loop system (\ref{math: closed loop system}) that has a pole at the origin.
\end{corollary}
\begin{proof}
Let $\beta$ be chosen such that $x = A^{-1}B_2\beta \neq 0$.
Hence, $(A - B_2(C_1B_2)^{-1}C_1A)x = B_2\beta - B_2\beta = 0$. Therefore, the Schur decomposition used in Theorem~\ref{lemma: mabrok synthesis lemma} results in a matrix $A_f$ such that $A_{11}$ has all of its eigenvalues in the closed left half plane and is singular. Also, $A_{22}$ is an anti-stable matrix.
\\\\
The closed-loop plant matrix is given by,
\begin{align*}
 & A + B_2K, \\
 =& A + B_2(C_1B_2)^{-1}(B_1^TP - C_1A - R(B_2^TC_1^T)^{-1}B_2^TP), \\
 =& UA_fU^T + B_2(C_1B_2)^{-1}(B_1^T - R(B_2^TC_1^T)^{-1}B_2^T)P, \\
 =& UA_fU^T + \hat{B}UP_fU^T,
\end{align*}
where $\hat{B} = B_2(C_1B_2)^{-1}(B_1^T - R(B_2^TC_1^T)^{-1}B_2^T)$.
To show that the closed-loop system has a pole at the origin we will construct a non-zero vector $w$ such that $(A + B_2K)Uw = 0$.
\\\\
We suppose $w$ has the form $w = [\tilde{w}^T 0]^T$ where $A_{11}\tilde{w}=0$. Then $A_fw$ becomes
\begin{align*}
 A_fw = \begin{bmatrix}
    \begin{tabular}{ l  r }
  $A_{11}$ & $A_{12}$ \\
  0 & $A_{22}$
\end{tabular}
\end{bmatrix}
 \begin{bmatrix}
  \begin{tabular}{ c}
  $\tilde{w}$ \\
  0
\end{tabular}
\end{bmatrix} =
\begin{bmatrix}
  \begin{tabular}{ c}
  $A_{11}\tilde{w}$ \\
  0
\end{tabular}
\end{bmatrix} = 0.
\end{align*}
The vector $P_fw$ is then
\begin{align*}
  P_fw = \begin{bmatrix}
    \begin{tabular}{ l  r }
  0 & 0 \\
  0 & $(T-S)^{-1}$
\end{tabular}
\end{bmatrix}\begin{bmatrix}
  \begin{tabular}{ c}
  $\tilde{w}$ \\
  0
\end{tabular}
\end{bmatrix} = 0.
\end{align*}
Therefore $A_fw + \hat{B}P_fw = 0$. Hence, the closed-loop plant matrix $(A + B_2K)$ satisfies $(A + B_2K)Uw = U(A_f + \hat{B}P_f)w = 0$ and has an eigenvalue at the origin.
\end{proof}

Theorem~\ref{lemma: mabrok synthesis lemma} gives a sufficient condition for synthesizing a controller $K$ that results in a closed-loop system with the NI property that is marginally stable. In order to ensure that the closed-loop system is asymptotically stable, \cite{Mabrok2012a,Mabrok2015} propose to apply a perturbation to the plant matrix $A$ in order to shift the poles by $\epsilon > 0$ to the right in the complex plane. The new plant matrix $A_{\epsilon} = A + \epsilon I$ is used in place of $A$ when designing the state feedback controller. This means that when the controller is applied to the actual system, the closed-loop system will have all its poles shifted left in the complex plane by $\epsilon$. This approach ensures the closed-loop system is asymptotically stable.
\\\\
Perturbation of the plant matrix does produce an asymptotically stable closed-loop system. However, the preservation of the NI property after perturbation was not guaranteed in \cite{Mabrok2012,Mabrok2012a,Mabrok2015}. A worked example showing a single successful case is given in \cite{Mabrok2012} (see also \cite{Mabrok2012a,Mabrok2015}) but no general proof was given. This issue will be addressed as part of our main result in the following section.

\section{The Main Result} \label{sec: main_result}

The following section presents our main result. A result is given showing that perturbation of the plant matrix of a SISO NI transfer function matrix does not change the NI property of the system for all perturbations $\epsilon \geq 0$. This is then followed by a method for the synthesis of an asymptotically stable closed-loop NI system that offers the computational advantages of Theorem~\ref{lemma: mabrok synthesis lemma}, in addition to achieving a closed-loop with a prescribed degree of stability.

\begin{lemma}\label{thrm: NI iff NI over entire orthant}
A SISO proper real rational transfer function $G(s)$ is negative imaginary if and only if $G(s)$ is analytic in $Re[s]>0$ and the inequality
\begin{align} \label{ep_0 NI condition}
    j(G(\sigma) - G(\sigma)^*) \geq 0
\end{align}
is satisfied for all $\sigma = j\omega + \epsilon$ which is not a pole of $G(s)$ where $\omega \geq 0$, $\epsilon \geq 0$.
\end{lemma}

\begin{proof}
Suppose the SISO transfer function $G(s)$ is proper, real, rational and NI. Hence, it satisfies the conditions of Definition~\ref{def:NI}. Therefore $G(s)$ is analytic in $Re[s] > 0$. Also since $G(s)$ is SISO, it is automatically symmetric. Then using Lemma 3.1 of \cite{FERRANTE20132138}, it follows that (\ref{ep_0 NI condition}) is satisfied.
\\\\
Conversely suppose $G(s)$ is analytic in $Re[s] > 0$ and (\ref{ep_0 NI condition}) is satisfied, Since $G(s)$ is SISO, proper and real rational, it satisfies the conditions of Lemma 3.1 of \cite{FERRANTE20132138} and hence the conditions of Definition~\ref{def:NI} are satisfied.
\end{proof}

\begin{remark}
A MIMO generalization of Lemma~\ref{thrm: NI iff NI over entire orthant} with a symmetry constraint can be obtained using the results in \cite{FERRANTE20132138,7317760}.
\\\\
Note that Corollary~\ref{corr: NI purturb still NI} does not generalize to the MIMO case unless symmetry of $G(s)$ is imposed.
\\\\
To show this fact, consider the following non-symmetric MIMO system $G(s) = C(sI-A)^{-1}B+D$ where,
\begin{align*}
  A &= \begin{bmatrix}
    \begin{tabular}{ l  r }
  $-1$ & $1$ \\
  0 & $-1$
\end{tabular}
\end{bmatrix}, &
  B &= \begin{bmatrix}
    \begin{tabular}{ l  r }
  $1$ & $-1$ \\
  0 & $1$
\end{tabular}
\end{bmatrix}, \\
  C &= \begin{bmatrix}
    \begin{tabular}{ l  r }
  $1$ & $0$ \\
  0 & $1$
\end{tabular}
\end{bmatrix}, &
D &= 0.
\end{align*}
\end{remark}

$G(s)$ is SNI via Lemma~\ref{lemma: SNI Riccati conditions}.
Suppose the plant matrix $G(s)$ is perturbed with a value of $\epsilon = 3$. In this case there is no longer a positive-definite solution to (\ref{math:riccati}) and therefore $G_{\epsilon}(s)$ is not NI. Thus, we have shown that Corollary~\ref{corr: NI purturb still NI} does not generalize to non-symmetric MIMO systems.
\\\\
The following corollary relates the previous lemma to perturbations in the plant matrix of an NI system.

\begin{corollary}\label{corr: NI purturb still NI}
If a SISO proper, real, rational transfer function matrix $G(s)$ with minimal state space realization $
\begin{bmatrix}
    \begin{tabular}{ l | r }
  A & B \\ \hline
  C & D
\end{tabular}
\end{bmatrix}
$ is NI, then the perturbed transfer function $G_{\epsilon}(s)$ with state space realization $
\begin{bmatrix}
    \begin{tabular}{ l | r }
  $A-\epsilon$ I & B \\ \hline
  C & D
\end{tabular}
\end{bmatrix}
$ will be NI for all $\epsilon \geq 0$.
\end{corollary}

\begin{proof}
The state space model of the perturbed system is given by
\begin{align}
\dot{x} &= (A - \epsilon I)x + Bu \\
y &= Cx + Du.
\end{align}
Hence
\begin{align}
G_\epsilon(j\omega) = C(j\omega I - A + \epsilon I)^{-1}B + D = G(j\omega +\epsilon).
\end{align}
This is exactly the $G(\sigma)$ from Lemma~\ref{thrm: NI iff NI over entire orthant}. Therefore, if $G(s)$ is NI, then $G_\epsilon(s) = G(s+\epsilon)$ is NI for all $\epsilon \geq 0$. \\
\end{proof}

\begin{remark}
The set of SNI transfer functions is not an open set. \\
\end{remark}

To establish this fact, consider the following SISO transfer function
\begin{align} \label{neg epsilon example}
  G(s) = \frac{s+1}{(s+2)(s+2)},
\end{align}
which has imaginary component
\begin{align*}
  Im[G(j\omega)] &= \frac{-\omega^3}{16\omega^2 + (4-\omega^2)^2} < 0 \quad \forall \omega > 0.
\end{align*}
This is clearly SNI. Now consider the perturbed transfer function $G_{\epsilon}(s) = G(s+\epsilon)$ with imaginary component
\begin{align*}
  Im[G_{\epsilon}(j\omega)] &= \frac{-2\epsilon - \epsilon^2 -\omega^2}{((2+\epsilon)^2 -\omega)^2 + 4\omega^2(2+\epsilon))^2}.
\end{align*}
The imaginary component of $G_{\epsilon}(s)$ is positive when $-2\epsilon - \epsilon^2 -\omega^2 > 0$ holds.
Thus, $G_{\epsilon}(s)$ does not have the NI property for any $\epsilon$ such that $-2 < \epsilon < 0$ holds. Therefore, the set of SNI transfer functions is not an open set (unlike the set of strictly bounded real transfer functions).
\\\\
Corollary~\ref{corr: NI purturb still NI} is now used to extend Theorem~\ref{lemma: mabrok synthesis lemma} to a result that guarantees preservation of the NI property for an asymptotically stable closed-loop system with a prescribed degree of stability.

\subsection{Schur Decomposition}
Let the constant $\epsilon>0$ defining the required stability margin be given. We begin by using a Schur decomposition of the matrix $A+\epsilon I -B_2(C_1B_2)^{-1}C_1(A+\epsilon I)$ as follows:

\begin{align*}
  A_f &= U^T(A+\epsilon I -B_2(C_1B_2)^{-1}C_1A  \\
  &-\epsilon B_2(C_1B_2)^{-1}C_1 )U =
  \begin{bmatrix}
    \begin{tabular}{ l  r }
  $A_{11}$ & $A_{12}$ \\
  0 & $A_{22}$
\end{tabular}
\end{bmatrix},
\\
  B_f &= U^T(B_2(C_1B_2)^{-1} - B_1R^{-1}) =
  \begin{bmatrix}
    \begin{tabular}{c}
      $B_{f1}$ \\
      $B_{f1}$
\end{tabular}
\end{bmatrix},\\
  \tilde{B_1} &= U^TB_1 =
  \begin{bmatrix}
    \begin{tabular}{c}
      $B_{11}$ \\
      $B_{22}$
\end{tabular}
\end{bmatrix},
\end{align*}
where $U$ is an orthogonal matrix obtained through the real Schur transformation; e.g see Section 5.4 of \cite{Bernstein2009}. As in \cite{Mabrok2012}, this decomposition allows the computational difficulties associated with singular Hamiltonians to be avoided.

\begin{theorem}\label{theorem: Ricc SNI synthesis}
Consider the LTI system (1) with a SISO, rational transfer function matrix $G(s)$ that has a minimal state space realization $
\begin{bmatrix}
    \begin{tabular}{ l | r }
  A & B \\ \hline
  C & D
\end{tabular}
\end{bmatrix}
$. For a given $\epsilon > 0$, there exists a static state-feedback matrix $K$ such that the closed-loop system (\ref{math: closed loop system}) is NI with degree of stability $\epsilon$ if there exist matrices $T \geq 0$ and $S \geq 0$ such that
\begin{align}
  -A_{22}T - TA_{22}^T + B_{f2}RB_{f2}^T &= 0, \label{math: T cond}\\
  -A_{22}S - SA_{22}^T + B_{22}R^{-1}B_{22}^T &= 0, \label{math: S cond}\\
  T-S &> 0,
\end{align}
where the matrices $A_{22}$,$B_{f2}$ and $B_{22}$ are obtained from the Schur decomposition given above.
\\\\
Furthermore, a corresponding state feedback controller matrix $K$ is given by,
\begin{align*}
K &= (C_1B_2)^{-1}(B_1^TP - C_1A - \epsilon C_1 - R(B_2^TC_1^T)^{-1}B_2^TP),
\end{align*}
where $P = UP_fU^T$ and $P_f = \begin{bmatrix}
    \begin{tabular}{ l  r }
  0 & 0 \\
  0 & $(T-S)^{-1}$
\end{tabular}
\end{bmatrix} \geq 0$.
Also, $P_f$ satisfies the algebraic Riccati equation
\begin{align*}
P_fA_f + A_f^TP_f - P_fB_fRB_f^TP_f + P_f\tilde{B_1}R^{-1}\tilde{B_1^T}P_f = 0.
\end{align*}
\end{theorem}

\begin{proof} Let $\epsilon > 0$ be given and suppose there exist matrices $T,S \geq 0$ satisfying equations (\ref{math: T cond}), (\ref{math: S cond}) such that $T>S$.
Subtracting (\ref{math: T cond}) from (\ref{math: S cond}) gives
\begin{align}
  A_{22}X + XA_{22}^T - B_{22}R^{-1}B_{22}^T + B_{f2}RB_{f2}^T = 0, \label{math: X equation}
\end{align}
where $X = S - T < 0$.
\\\\
Let $P_1 = -X^{-1} > 0$ and pre and post multiply (\ref{math: X equation}) by $X^{-1}$ to get the following Riccati equation:
\begin{multline}
  P_1A_{22} + A_{22}^TP_1 - P_1B_{22}R^{-1}B_{22}^TP_1 + \\ P_1B_{f2}RB_{f2}^TP_1 = 0.
\end{multline}
It follows that
\begin{align}
  P_fA_{f} + A_{f}^TP_f - P_f\tilde{B_{1}}R^{-1}\tilde{B_{1}}^TP_f + P_fB_{f}RB{f}^TP_f = 0, \label{math: X equation2}
\end{align}
has a solution $P_f = \begin{bmatrix}
    \begin{tabular}{ l  r }
  0 & 0 \\
  0 & $P_1$
\end{tabular}
\end{bmatrix} \geq 0$, where $A_f$,$B_f$ and $\tilde{B_{1}}$ are defined as above.
\\\\
After some algebraic manipulation, (\ref{math: X equation2}) can be written as
\begin{align}
  P\tilde{A} + \tilde{A}^TP + PB_1R^{-1}B_1^TP + Q = 0, \label{math: final NI synth equation}
\end{align}
where
\begin{align*}
  \tilde{A} &= A+\epsilon I - B_1R^{-1}C_1(A+\epsilon I) + (I - B_1R^{-1}C_1)B_2K, \\
            &= A_{cl} - B_1R^{-1}C_1A_{cl}, \\
  R & = C_1B_1 + B_1^TC_1^T, \\
  Q &= (A+\epsilon I + B_2K)^TC_1^TR^{-1}C_1(A+\epsilon I + B_2K), \\
    &= A_{cl}^TC_1^TR^{-1}C_1A_{cl}.
\end{align*}
Here $A_{cl} = A+\epsilon I + B_2K$ is the perturbed plant matrix of the closed-loop system (\ref{math: closed loop system}). Therefore, since $P = U\begin{bmatrix}
    \begin{tabular}{ l  r }
  0 & 0 \\
  0 & $P_1$
\end{tabular}
\end{bmatrix}U^T \geq 0$ is a solution to (\ref{math: final NI synth equation}), then it follows from Lemma~\ref{theorem:NI using riccati} that the perturbed closed-loop system is NI. Also, Corollary~\ref{corollary: closed-loop synthesises system has origin pole} implies that the perturbed closed-loop transfer function will have a pole at the origin. However, the actual closed-loop system will have have its poles shifted by an amount $\epsilon$ to the left in the complex plane resulting in the desired asymptotically stable system with degree of stability $\epsilon$. Further to this, it follows from Lemma~\ref{thrm: NI iff NI over entire orthant} that the actual closed-loop system (\ref{math: closed loop system}) is NI.
\end{proof}

\begin{remark}
  The closed-loop system using the state feedback controller $K$ synthesis in Theorem~\ref{theorem: Ricc SNI synthesis} will have a pole located at $-\epsilon$. All of the remaining closed-loop poles will be to the left of this pole.
\end{remark}

The theorem presented above can be used to synthesize a controller that results in an asymptotically stable closed-loop system with a prescribed degree of stability. Further to this, Theorem~\ref{theorem: Ricc SNI synthesis} extends Lemma~\ref{lemma: mabrok synthesis lemma} by guaranteeing the closed-loop system also has the NI property. For higher order systems, this synthesis approach also offers a computational advantage over alternative LMI based techniques in that the solution can be obtained from a Schur decomposition and two Lyapunov equations.

\section{Illustrative Example}

The following section illustrates how Theorem~\ref{theorem: Ricc SNI synthesis} may be applied. Consider the following uncertain system of the form (4), (5), (6) considered in \cite{Mabrok2012}:

\begin{align*}
  A = \begin{bmatrix}
      \begin{tabular}{ l c r }
      -1 & 0 & -1 \\
       1 & 1 & -1 \\
       -5 & 1 & 1
    \end{tabular}
  \end{bmatrix},
  B_1 = \begin{bmatrix}
      \begin{tabular}{ c }
       -1   \\
       1   \\
       0
    \end{tabular}
  \end{bmatrix},  \\
  B_2 = \begin{bmatrix}
      \begin{tabular}{ c }
       0   \\
       4   \\
       2
    \end{tabular}
  \end{bmatrix},
  C_1 = \begin{bmatrix}
      \begin{tabular}{ l c r }
       0 & 2 & -3
    \end{tabular}
  \end{bmatrix}.
\end{align*}

The example given in \cite{Mabrok2012a} uses the synthesis technique outlined in Lemma~\ref{lemma: mabrok synthesis lemma} and suggests a perturbation of $\epsilon = 0.3$ in order to move any poles away from the origin. It follows from Corollary~\ref{corr: NI purturb still NI} and Theorem~\ref{thrm: NI iff NI over entire orthant} that any $\epsilon \geq 0$ will result in a NI closed-loop system provided that the condition $T>S$ is satisfied.
\\\\
Applying the Schur decomposition to the matrix $A+\epsilon I-B_2(C_1B_2)^{-1}C_1A-\epsilon B_2(C_1B_2)^{-1}C_1$ for any value of $\epsilon > 0$ results in a matrix $A_f$ with the following form
\begin{align*}
  A_f = U^T \bigg(\begin{bmatrix}
      \begin{tabular}{ l c r }
      -1 & 0 & -1 \\
      -33 & 6 & 9 \\
      -22 & 4 & 6
    \end{tabular}
  \end{bmatrix} + \epsilon \begin{bmatrix}
      \begin{tabular}{ l c r }
      1 & 0 & 0 \\
      0 & -3 & 6 \\
      0 & -2 & 4
    \end{tabular}
  \end{bmatrix}\bigg)U,
\end{align*}
where $U$ is the real Schur transformation matrix and is dependent on the value chosen for $\epsilon$.
\\\\
We will now choose $\epsilon = 2$.
\\\\
The solution to the Lyapunov equations (\ref{math: T cond}), (\ref{math: S cond}) gives $T = 0.039$ and $S = 0.019$ which implies that
\begin{align*}
X = -0.020.
\end{align*}

It follows that
\begin{align}
  P_f = \begin{bmatrix}
      \begin{tabular}{ l c r }
      0 & 0 & 0 \\
     0 & 0 & 0 \\
     0 & 0 & $49.078$
    \end{tabular}
  \end{bmatrix} \geq 0.
\end{align}
Therefore, $P = UP_fU^T \geq 0$ is a solution to (\ref{math: final NI synth equation}) and the controller gain matrix is given by
\begin{align}
  K = \begin{bmatrix}
      \begin{tabular}{ l c r }
      34.008 & -15.984 & 0.680
    \end{tabular}
  \end{bmatrix}.
\end{align}

The closed-loop system formed using this state feedback controller is NI with real poles located at -2.0, -2.5 and -66.1 in the complex plane. However, it follows from Theorem~\ref{theorem: Ricc SNI synthesis} that the closed loop system is NI for all values of $\epsilon \geq 0$ and the closed-loop system is asymptotically stable with a pole at $-\epsilon$ and all of the other closed-loop poles to the left of this pole.
\\\\
The preservation of the NI property for $\epsilon \geq 0$ can be seen for this system in Figure~\ref{fig_worked_example_bode} which shows the Bode diagram of the perturbed closed-loop transfer function for a perturbation value of $\epsilon = 2$.

\begin{figure}[h!]
\centering
\includegraphics[width=3.3in]{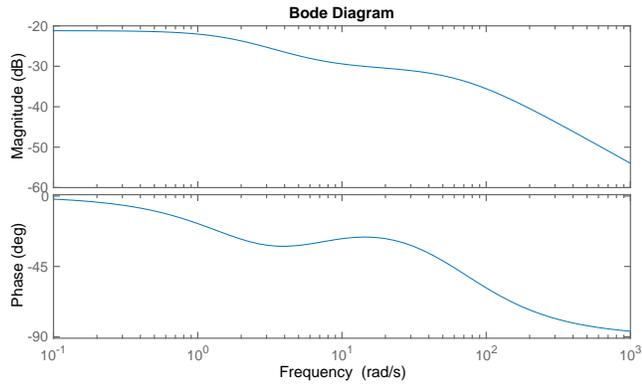}
\caption{Bode diagram of the closed-loop transfer function for a perturbation values of $\epsilon = 2$. Note that the closed loop system is actually SNI.}
\label{fig_worked_example_bode}
\end{figure}

\section{CONCLUSIONS}

This paper has shown that a SISO system that satisfies the negative imaginary property will maintain the negative imaginary property for all positive perturbations applied to the plant matrix. This result was used to develop a new method for negative imaginary controller synthesis that extends existing techniques by guaranteeing a closed-loop system that is asymptotically stable with a prescribed degree of stability and the NI property. The synthesis approach used, relies on the solution to two Lyapunov equations and as such offers computational advantages for higher order systems when compared with alternative LMI synthesis approaches.


\bibliographystyle{IEEEtran}
\bibliography{root}


\end{document}